\documentclass[11pt]{amsart}
\usepackage[utf8]{inputenc}
\usepackage{amsmath}
\usepackage{amsfonts}
\usepackage{amssymb}
\usepackage{tikz-cd}
\usepackage{enumitem}
\usepackage{stmaryrd}
\usetikzlibrary{decorations.markings}
\tikzset{double line with arrow/.style args={#1,#2}{decorate,decoration={markings,%
mark=at position 0 with {\coordinate (ta-base-1) at (0,1pt);
\coordinate (ta-base-2) at (0,-2pt);},
mark=at position 1 with {\draw[#1] (ta-base-1) -- (0,1pt);
\draw[#2] (ta-base-2) -- (0,-2pt);
}}}}
\tikzset{Equal/.style={-,double line with arrow={-,-}}}

\newcommand\R{\mathbb{R}}

\newcommand{\HS}{\mathrm{H}}

\newcommand{\CS}{\mathrm{S}}
\newcommand{\CPL}{\mathrm{C}}

\newcommand{\RU}{\underline{\R}}
\newcommand{\kU}{\underline{k}}
\newcommand{\OmegaU}{{}^+\Omega}

\newcommand{\FC}{\mathcal{F}}
\newcommand{\GC}{\mathcal{G}}
\newcommand{\KC}{\mathcal{K}}
\newcommand{\IC}{\mathcal{I}}
\newcommand{\JC}{\mathcal{J}}

\newcommand{\Id}{\mathrm{id}}
\newcommand{\mult}{\mathrm{mult}}
\newcommand{\const}{\mathrm{const}}
\newcommand{\Lip}{\mathrm{Lip}}

\newcommand{\Sing}{\mathrm{Sing}\,}

\newcommand{\HH}{\mathbb{H}}

\newcommand{\sh}{\mathrm{sh}}

\newcommand{\spec}{\mathrm{spec} _{\mathbb{R}}}

\newtheorem{theorem}{Theorem}
\newtheorem*{theorem*}{Theorem}

\newtheorem{corollary}[theorem]{Corollary}
\newtheorem{lemma}[theorem]{Lemma}

\theoremstyle{definition}
\newtheorem*{definition}{Definition}

\theoremstyle{remark}

\let\<\langle
\let\>\rangle

\let\uml\"

\title[]{The splitting of the de Rham cohomology\\of soft function algebras is multiplicative} 

\author[]{Igor Baskov}

\address{} 
\thanks{This research was partially financed by the Russian Science Foundation via agreement ${\rm N\textsuperscript{\underline{o}}}$ 24-21-00119.
}


\begin{document}
\newcommand{\Addresses}{{
  \bigskip
  \footnotesize
\noindent
  \textsc{St. Petersburg Department of Steklov Mathematical Institute\\of Russian Academy of Sciences}\par\noindent
  \textit{E-mail address}: \texttt{baskovigor@pdmi.ras.ru}
}}

\begin{abstract}
Let $A$ be a real soft function algebra.
In \cite{baskov23} we have obtained a canonical splitting $\HS ^* (\Omega ^\bullet _{A|\R}) \cong \HS ^* (X,\R)\oplus \text{(something)}$ via the canonical maps $\Lambda_A:\HS ^* (X,\R)\to\HS ^* (\Omega ^\bullet _{A|\R})$ and $\Psi_A:\HS ^* (\Omega ^\bullet _{A|\R})\to\HS ^* (X,\R)$.
In this paper we prove that these maps are multiplicative.
\end{abstract}

\maketitle
\section{Introduction}

All algebras are assumed to be commutative.
To an algebra $A$ over a field $k$ one associates a dg-algebra $\Omega ^\bullet _{A|k}$ with $\Omega ^0 _{A|k} = A$, called the de Rham dg-algebra, in a standard way (see \cite[$\S 3$]{KunzKahlerDifferentials}).

In this paper we extend the results of the paper \cite{baskov23}.
There, for a soft sheaf of $k$-algebras $\FC$ on a compact Hausdorff space $X$ we have constructed a linear map
$$\Lambda _\FC :\HH ^* (X,\kU _X [0])\to \HS ^* (\Omega ^\bullet _{\FC (X)|k}).$$
Here the domain is the cohomology of $X$ with coefficients in the constant sheaf $\kU _X$.
This map is natural with respect to morphisms of $k$-ringed spaces.
For an arbitrary space $X$ and a subalgebra $A\hookrightarrow C(X)$ of the $\R$-algebra of real-valued continuous functions on $X$	 we have constructed a linear map
$$\Psi _A:\HS ^* (\Omega ^\bullet _{A|\R})\to \HH ^* (X,\RU _X [0]).$$
This map is natural with respect to continuous maps of spaces covered by a homomorphism of algebras.
We have proved that for a compact Hausdorff space $X$ and a soft subsheaf $\FC$ of $C_X$, the sheaf of $\R$-algebras of real-valued continuous functions, the composition $\Psi _{\FC (X)} \circ\Lambda _\FC$ is the identity map.
Thus, the groups $\HH ^* (X,\RU _X [0])$ canonically split off of $\HS ^* (\Omega ^\bullet _{\FC (X)})$.

But there is still a question of whether the maps $\Lambda _\FC$ and $\Psi _A$ are multiplicative with respect to the cup product on cohomology. The answer is yes and is given by Theorem~\ref{thm:lambdaismultiplicative} and Theorem~\ref{thm:psiismultiplicative}.

\subsection*{Acknowledgements}
I would like to thank Dr. Sem{\"e}n Podkorytov for his immense patience, along with many fruitful discussions and his invaluable help in drafting this paper.
I am grateful to the St. Petersburg Department of Steklov Mathematical Institute of Russian Academy of Sciences for their financial assistance.

\section{Cup product in hypercohomology}\label{sec:productinhypcoh}

\subsection{Hypercohomology groups.}
By a complex we mean a non-negative cochain complex.
By default, a complex consists of $k$-vector spaces.
For a sheaf $\FC$ of vector spaces over a field $k$ we denote by $\FC [0]$ the complex of sheaves with $\FC$ in degree $0$ and other terms zero.
We denote by $\kU_X$ the constant sheaf on a space $X$ associated with the field $k$.

For any complex $\FC^\bullet$ there is a complex $\IC^\bullet$ of injective sheaves and a quasi-isomorphism $i:\FC^\bullet\to\IC^\bullet$, see \cite[Proposition $8.4$]{Voisin_2002}.
We call $\IC^\bullet$ an \textit{injective resolution} of $\FC^\bullet$.

One defines the \textit{hypercohomology} of a complex of sheaves $\FC^\bullet$ as
$$\HH^*(X,\FC^\bullet):=\HS^*(\IC^\bullet(X))$$
for some injective resolution $\FC^\bullet\to\IC^\bullet$.
We refer the reader to \cite[Definition~$10.2$]{wedhorn2016manifolds} for the precise definition of the hypercohomology groups.

There is a canonical homomorphism
$$\Upsilon :\HS ^* (\FC ^\bullet (X))\to\HH ^* (X,\FC ^\bullet),$$
natural with respect to morphisms of complexes of sheaves.

Take two complexes of sheaves $\FC ^\bullet$ and $\GC ^\bullet$ on topological spaces $X$ and $Y$, respectively.
Suppose $\phi:Y\to X$ is a continuous map and $\gamma :\FC ^\bullet \to \phi_* \GC ^\bullet$ is a morphism of complexes of sheaves.
Then there is the induced map on the hypercohomology $\HH^*(\phi,\gamma) :\HH^*(X,\FC ^\bullet )\to \HH^*(Y,\GC ^\bullet )$.

\subsection{External cup product.}
For two complexes of sheaves $\FC^\bullet$ and $\GC^\bullet$ on a space $X$ one defines their tensor product over $k$, $\FC^\bullet\otimes\GC^\bullet$, as a complex of sheaves on $X$, see \cite[Chapter II, $6.1$]{godement1958topologie}.
We construct a map 
$$\HH^*(X,\FC^\bullet)\otimes \HH^*(X,\GC^\bullet)\to \HH^*(X,\FC^\bullet\otimes \GC^\bullet),$$
called the \textit{external cup product}.
Choose injective resolutions $$i_1:\FC^\bullet\to \IC^\bullet,\;\;\; i_2:\GC^\bullet\to \JC^\bullet,\;\;\;j:\FC^\bullet\otimes \GC^\bullet\to \KC^\bullet.$$
The map of complexes
$$i_1\otimes i_2:\FC^\bullet\otimes\GC^\bullet\to\IC^\bullet\otimes\JC^\bullet$$
is a quasi-isomorphism, as the tensor product of sheaves over a field is an exact functor.
There exists a unique up to homotopy morphism of complexes $\beta:\IC^\bullet\otimes\JC^\bullet\to\KC^\bullet$ making the following diagram commute up to homotopy
\begin{equation*}
\begin{tikzcd}[row sep=large,column sep = large]
\FC^\bullet\otimes\GC^\bullet\arrow[d,"j"]\arrow[r,"i_1\otimes i_2"]&\IC^\bullet\otimes\JC^\bullet\arrow[dl,dotted,"\beta"]\\
\KC^\bullet,&
\end{tikzcd}
\end{equation*}
see \cite[Lemma~$13.18.6$ and Lemma~$13.18.7$]{stacks-project}.

Taking the global sections over $X$ one obtains the natural map
$$\beta(X):\left(\IC^\bullet\otimes\JC^\bullet\right) (X)\to \KC^\bullet(X).$$
On the other hand, one has the natural map
$$\IC^\bullet(X)\otimes\JC^\bullet(X)\to\left(\IC^\bullet\otimes\JC^\bullet\right)(X).$$
Composing these maps and taking the cohomology leads to the natural map
$$\HH^*(X,\FC^\bullet)\otimes\HH^*(X,\GC^\bullet)\xrightarrow{}\HH^*(X,\FC^\bullet\otimes\GC^\bullet).$$
This product is natural in $\FC^\bullet$ and $\GC^\bullet$.
Moreover, it is natural in triples $(X,\FC^\bullet, \GC^\bullet)$.

The external cup product is compatible with $\Upsilon$, that is, the following diagram is commutative:
\begin{equation}\label{eq:extcupiscompatiblewithupsilon}
\begin{tikzcd}
\HS ^* (\FC ^\bullet (X))\otimes\HS ^* (\GC ^\bullet (X))\ar[r]\ar[d,"\Upsilon\otimes\Upsilon"]&\HS ^* ((\FC ^\bullet\otimes\GC^\bullet)(X))\ar[d,"\Upsilon"]\\
\HH^*(X,\FC^\bullet)\otimes \HH^*(X,\GC^\bullet)\ar[r,"\cup"]&\HH^*(X,\FC^\bullet\otimes\GC^\bullet),
\end{tikzcd}
\end{equation}
where the top map is induced by the natural map of complexes $$\FC^\bullet(X)\otimes\GC^\bullet(X)\to(\FC^\bullet\otimes\GC^\bullet)(X).$$
\subsection{Internal cup product.}
Let $\FC^\bullet$ be a sheaf complex with multiplication, that is, a complex of sheaves $\FC^\bullet$ comes with a morphism of complexes
$$\FC^\bullet\otimes\FC^\bullet\xrightarrow{\times}\FC^\bullet,$$
called the multiplication.
Composing with the external cup product, this multiplication gives rise to the product
$$\HH^*(X,\FC^\bullet)\otimes\HH^*(X,\FC^\bullet)\xrightarrow{\cup}\HH^*(X,\FC^\bullet),$$
which we call the (\textit{internal}) \textit{cup product}.

The cup product on hypercohomology satisfies the following properties.
\begin{enumerate}
\item\label{prop:naturalityofproduct}
The product on the hypercohomology
$$\HH^*(X,\FC^\bullet)\otimes \HH^*(X,\FC^\bullet)\xrightarrow{\cup} \HH^*(X,\FC^\bullet)$$
is natural in the following sense.
Let $\gamma:\FC^\bullet\to\GC^\bullet$ be a morphism of sheaf complexes with multiplication, that is, the diagram
\begin{equation*}
\begin{tikzcd}
\FC^\bullet\otimes\FC^\bullet\ar[r,"\times"]\ar[d,"\gamma\otimes\gamma",swap]&\FC^\bullet\ar[d,"\gamma"]\\
\GC^\bullet\otimes\GC^\bullet\ar[r,"\times"]&\GC^\bullet
\end{tikzcd}
\end{equation*}
is commutative.

Then we have the commutative diagram
\begin{equation*}
\begin{tikzcd}
\HH^*(X,\FC^\bullet)\otimes \HH^*(X,\FC^\bullet)\ar[r,"\cup"]\ar[d,"\HH^*(\gamma)\otimes\HH^*(\gamma)",swap]&\HH^*(X,\FC^\bullet)\ar[d,"\HH^*(\gamma)"]\\
\HH^*(X,\GC^\bullet)\otimes \HH^*(X,\GC^\bullet)\ar[r,"\cup"]&\HH^*(X,\GC^\bullet).
\end{tikzcd}
\end{equation*}

\item\label{prop:functorialitydifferentspaces}
Take a space $X$ with a sheaf complex with multiplication $\FC^\bullet$ and a space $Y$ with a sheaf complex with multiplication $\GC^\bullet$.
Let $\phi:Y\to X$ be a continuous map and $\gamma:\FC^\bullet\to\phi_*\GC^\bullet$ be a morphism of sheaf complexes with multiplication.
Then we have the commutative diagram
\begin{equation*}
\begin{tikzcd}
\HH^*(X,\FC^\bullet)\otimes \HH^*(X,\FC^\bullet)\ar[r,"\cup"]\ar[d,"\HH^*(\phi{,}\gamma)\otimes\HH^*(\phi{,}\gamma)",swap]&\HH^*(X,\FC^\bullet)\ar[d,"\HH^*(\phi{,}\gamma)"]\\
\HH^*(Y,\GC^\bullet)\otimes \HH^*(Y,\GC^\bullet)\ar[r,"\cup"]&\HH^*(Y,\GC^\bullet).
\end{tikzcd}
\end{equation*}

\item\label{prop:compatiblewithupsilon}
The product on hypercohomology is compatible with $\Upsilon$.

Explicitly, the multiplication on $\FC^\bullet$ induces the product
$$\HS ^* (\FC ^\bullet (X))\otimes \HS ^* (\FC ^\bullet (X))\to\HS ^* (\FC ^\bullet (X)).$$
Then the following diagram is commutative:
\begin{equation*}\label{eq:multiplicationhypercohomologycoincideswithusual}
\begin{tikzcd}
\HS ^* (\FC ^\bullet (X))\otimes\HS ^* (\FC ^\bullet (X))\ar[r]\ar[d,"\Upsilon\otimes\Upsilon"]&\HS ^* (\FC ^\bullet (X))\ar[d,"\Upsilon"]\\
\HH^*(X,\FC^\bullet)\otimes \HH^*(X,\FC^\bullet)\ar[r,"\cup"]&\HH^*(X,\FC^\bullet).
\end{tikzcd}
\end{equation*}
\end{enumerate}

The sheaf complex $\kU_X[0]$ carries multiplication
$$\kU_X[0]\otimes \kU_X[0]\to \kU_X[0].$$
The resulting cup product 
$$\HH^*(X,\kU_X[0])\otimes\HH^*(X,\kU_X[0])\xrightarrow{\cup}\HH^*(X,\kU_X[0])$$
coincides with the usual cup product on the sheaf cohomology, see \cite[\S~$5.3.2$]{Voisin_2002}.

\section{Algebraic de Rham forms}

For any commutative algebra with unity $A$ over a field $k$ we denote by $\Omega^\bullet _{A|k}$ the dg-algebra of algebraic de Rham forms, see \cite[$\S 3$]{KunzKahlerDifferentials}.

Let $X$ be a compact Hausdorff space.
Consider a soft sheaf of algebras $\FC$ on $X$.
In \cite[Section $2.4$]{baskov23} we consider the complex of presheaves $\Omega^\bullet_{\FC|k}$ on $X$ with $\Omega^\bullet_{\FC|k}(U)=\Omega^\bullet_{\FC(U)|k}$.
For $n\geq 0$ we denote by ${}^+\Omega^n_{\FC|k}$ the associated sheaf of $\Omega^n_{\FC|k}$.
The sheaves ${}^+\Omega^n_{\FC|k}$ form a complex of sheaves ${}^+\Omega^\bullet_{\FC|k}$.
We denote by
$$\mathrm{sh}:\Omega^\bullet_{\FC(X)|k}=\Omega^\bullet_{\FC|k}(X)\to {}^+\Omega^\bullet_{\FC|k}(X)$$
the sheafification map.

We define a morphism of sheaf complexes with multiplication $\epsilon :\kU_X [0]\to {}^+\Omega^\bullet_{\FC|k}$, called the coaugmentation, by $\epsilon (1):=1$.

The wedge product
$$\Omega^\bullet_{A|k}\otimes \Omega^\bullet_{A|k}\xrightarrow{\wedge} \Omega^\bullet_{A|k}$$
for any $k$-algebra $A$ gives rise to a product
$$\Omega^\bullet_{\FC|k}\otimes \Omega^\bullet_{\FC|k}\xrightarrow{\wedge} \Omega^\bullet_{\FC|k}$$
for the complex of presheaves $\Omega^\bullet_{\FC|k}$.
Passing to the associated sheaves gives rise to a product
$${}^+\Omega^\bullet_{\FC|k}\otimes {}^+\Omega^\bullet_{\FC|k}\xrightarrow{\wedge} {}^+\Omega^\bullet_{\FC|k}.$$
Taking the global sections yields the commutative diagram
\begin{equation}\label{diag:shmultiplicative}
\begin{tikzcd}
\Omega^\bullet_{\FC(X)|k}\otimes\Omega^\bullet_{\FC(X)|k}\ar[r,"\wedge"]\ar[d,"\sh\otimes\sh",swap]&\Omega^\bullet_{\FC(X)|k}\ar[d,"\sh"]\\
{}^+\Omega^\bullet_{\FC|k}(X)\otimes{}^+\Omega^\bullet_{\FC|k}(X)\ar[r,"\wedge"]&{}^+\Omega^\bullet_{\FC|k}(X).
\end{tikzcd}
\end{equation}

Now, we have the cup product
$$\HH^*(X,{}^+\Omega^\bullet_{\FC|k})\otimes\HH^*(X,{}^+\Omega^\bullet_{\FC|k})\xrightarrow{\cup} \HH^*(X,{}^+\Omega^\bullet_{\FC|k}).$$

\section{The map $\Lambda_\FC$ is multiplicative}
In this section we briefly recall the construction of the natural map
$$\Lambda_\FC:\HH ^*(X,\kU_X[0])\to\HH ^*(X,{}^+\Omega^\bullet_{\FC|k})$$
constructed in \cite[Chapter $3$]{baskov23}.
Then we prove that the map $\Lambda_\FC$ is multiplicative.

\begin{definition}\label{def:lambdamap}
For a soft sheaf of algebras $\FC$ on a compact Hausdorff space $X$ we define the map $\Lambda_\FC$ by the following diagram:
\begin{equation*}
\begin{tikzcd}
\HH^* (X,\kU _X [0] )\arrow[r,"\HH^* (\epsilon )"]\arrow[ddr,"\Lambda_\FC"',dashed] & \HH^* (X,\OmegaU ^\bullet _{\FC|k})\\
 & \HS ^* (\OmegaU ^\bullet _{\FC|k} (X)) \arrow[u,"\Upsilon"',"\cong"]\\
 & \HS ^* (\Omega ^\bullet _{\FC (X)|k}) \arrow[u,"\cong","\HS^* (\sh )"'].
\end{tikzcd}
\end{equation*}
The maps $\Upsilon$ and $\HS^*(\sh )$ are isomorphisms, see \cite[Section~$3.2$]{baskov23}.

The map $\HS^*(\sh )$ is a multiplicative by Diagram~\ref{diag:shmultiplicative}.
The map $\Upsilon$ is multiplicative by the property~\ref{prop:compatiblewithupsilon} of the internal cup product.
The map $\HH^*(\epsilon)$ is multiplicative by the property~\ref{prop:naturalityofproduct} of the internal cup product.
\end{definition}

Hence, we obtain the following theorem:

\begin{theorem}\label{thm:lambdaismultiplicative}
The map
$$\Lambda_\FC :\HH ^* (X,\kU _X [0] )\to \HS ^* (\Omega ^\bullet _{\FC (X)|k})$$
is multiplicative.
\end{theorem}

\section{Simplicial techniques}

We call a fibrant pointed simplicial set $(K,o)$ \textit{connected} if for each $v\in K_0$ there exists a simplicial map $\Delta[1]\to K$ such that $0\mapsto o$ and $1\mapsto v$ (here $0,1\in \Delta[1]_0$ are the ends).
We call $(K,o)$ \textit{contractible} if there exists a basepoint preserving homotopy
$\Delta[1]\times K\to K$ that maps the bottom base to $o$ and is the identity on the top base.

A simplicial vector space $B$ is called a simplicial module over a simplicial algebra $A$ if there is a simplicial linear map
$$A\otimes B\xrightarrow{\times} B$$
that makes $B_p$ into an $A_p$-module for each $p\geq 0$,
cf.\ \cite[Chapter $2.6$]{quillen67}.

\begin{lemma}\label{lem:simplicialalgebraiscontractible}
Let $B$ be a simplicial module over a connected simplicial algebra $A$.
Then $B$ is contractible.
\end{lemma}

\begin{proof}
As $A$ is connected there exists a simplicial map $$\gamma :\Delta[1]\to A$$ connecting $0$ and $1$.
Then the composition
\begin{equation*}
\begin{tikzcd}
\Delta[1]\times B\ar[r,"\gamma\times\Id"]& A\times B\ar[r,"\times"]&B
\end{tikzcd}
\end{equation*}
is the required homotopy.
\end{proof}

We call a simplicial complex of vector spaces $A^\bullet$ \textit{degree-wise contractible} if $A^q$ is contractible for each $q\geq 0$.
Note that in \cite{felix01} the degree-wise contractible simplicial complexes are called ``extendable complexes''.
We call a simplicial map of simplicial complexes $A^\bullet\to B^\bullet$ a \textit{dimension-wise quasi-isomorphism} if it is a quasi-isomorphism in each simplicial dimension.

\begin{corollary}\label{cor:simplicialalgebraiscontractible}
Let $A^\bullet$ be a simplicial graded algebra such that the simplicial algebra $A^0$ is connected.
Then $A^\bullet$ is degree-wise contractible.
\end{corollary}

\begin{corollary}\label{cor:productisdegreewisecontractible}
Let $A^\bullet$ and $B^\bullet$ be simplicial graded algebras and $A^\bullet$ be degree-wise contractible.
Then $A^\bullet\otimes B^\bullet$ is degree-wise contractible.
\end{corollary}
\begin{proof}
The simplicial vector space $(A^\bullet\otimes B^\bullet)^q$ is a simplicial module over the simplicial algebra $A^0$, which is connected as it is contractible.
\end{proof}

For any simplicial complex $A^\bullet$ and a simplicial set $K$ one defines a complex $A^\bullet (K)$ in a usual way:
for $q\geq 0$ the vector space $A^q(K)$ consists of all simplicial maps $K\to A^q$, cf. \cite[Chapter $10$, $(b)$]{felix01}.
If $A^\bullet$ is a simplicial dg-algebra, then $A^\bullet(K)$ is a dg-algebra.

\begin{lemma}\label{lem:degreewisecontrinduces}
Let $\theta:A^\bullet \to B^\bullet$ be a dimension-wise quasi-isomorphism of degree-wise contractible simplicial complexes.
Then for any simplicial set $K$, the induced morphism of complexes $\theta(K):A^\bullet(K)\to B^\bullet(K)$ is a quasi-isomorphism.
\end{lemma}
\begin{proof}
See \cite[Proposition $10.5$]{felix01}. 
\end{proof}

\section{The map $\Psi$ is multiplicative}
In the paper~\cite{baskov23} we have contructed the map
$$\Psi:\HS^*(\Omega^\bullet_{A|\R})\to \HH^*(X,\underline{\R}_X[0])$$
for any topological space $X$ and any subalgebra $A$ of the algebra of real-valued continuous functions on $X$.
In this section we prove that this map is multiplicative.

\subsection{Simplicial dg-algebra $\R[0]_\const$.}
We denote by $\R[0]_\const$ the simplicial dg-algebra that equals $\R[0]$ in each dimension.
We have $$\R[0]_\const\otimes \R[0]_\const= \R[0]_\const.$$

\subsection{Simplicial dg-algebra $\CPL^\bullet$.}
We denote by $\CPL^\bullet$ the simplicial dg-algebra of (unnormalized) simplicial cochains on the combinatorial simplices $\Delta[n]$, cf. \cite[Chapter $10$, $(d)$]{felix01}.
The multiplication is the usual Alexander--Whitney cup product.
There is a unique morphism of simplicial dg-algebras $\epsilon:\R[0]_\const\to\CPL ^\bullet$, called the coaugmentation.

\begin{lemma}\label{lem:lipschitzcochainsaredwcontractible}
The simplicial dg-algebra $\CPL^\bullet$ is degree-wise contractible.
Moreover the coaugmentation map $\epsilon:\R[0]_\const\to\CPL ^\bullet$ is a dimension-wise quasi-isomorphism.
\end{lemma}
Cf. \cite[Lemma $10.12$]{felix01} for the normalized simplicial cochains.

\subsection{Simplicial dg-algebra $\Omega^\bullet _\flat$}
In \cite[Chapter $5$]{baskov23} we have introduced the simplicial dg-algebra of flat cochains $\Omega^\bullet _\flat$.
For each $n\geq 0$, the dg-algebra $\Omega^\bullet _\flat (\Delta ^n)$ consists of affine cochains on $\Delta ^n$ bounded with respect to the flat seminorm.

Consider the simplicial algebra $\Lip$, where, for each $n\geq 0$, the algebra $\Lip(\Delta^n)$ consists of all Lipschitz functions on $\Delta^n$.
By \cite[$5.2(2)$]{baskov23} there is a morphism of simplicial algebras $$\zeta :\Lip\to \Omega ^0 _\flat.$$

\begin{lemma}\label{lem:omegaisdegreewisecontractible}
The simplicial dg-algebra $\Omega^\bullet _\flat $ is degree-wise contractible.
\end{lemma}

\begin{proof}
By construction of $\Omega^\bullet _\flat $, for each $q\geq 0$, $\Omega^q _\flat $ is a simplicial module over $\Omega^0 _\flat $ and, hence, over $\Lip$.
By Lemma~\ref{lem:simplicialalgebraiscontractible} it suffices to prove that the simplicial algebra $\Lip$ is connected, which is easy to see.
\end{proof}

There is a unique morphism of simplicial dg-algebras $\epsilon:\R[0]_\const\to\Omega^\bullet _\flat (\Delta^n)$, called the coaugmentation.

\begin{lemma}\label{lem:omegaflatisacyclic}
The map $\epsilon:\R[0]_\const\to\Omega^\bullet _\flat$ is a dimension-wise quasi-isomorphism.
\end{lemma}
\begin{proof}
See the construction of $\Omega^\bullet _\flat (\Delta ^n)$ and \cite[Chapter VII, Theorem $12$A]{whitney57}.
\end{proof}

\subsection{Spectrum of a finitely generated algebra.}
For a set $Z\subset\R^n$ one defines the simplicial set $\Sing _\Lip (Z)$ of Lipschitz singular simplices.

We denote by $\CS^\bullet _\Lip(Z)$ the complex of Lipschitz singular cochains on $Z$.
This complex can be equipped with the Alexander--Whitney cup product, which is non-commutative.
We have
$$\CPL ^\bullet (\Sing _\Lip (Z))=\CS^\bullet _\Lip(Z),$$
where the product on the left, that comes from the product on $\CPL ^\bullet$, coincides with the Alexander--Whitney cup product on the right.

For a finitely generated $\R$-algebra $B$ one constructs the \textit{real maximal spectrum} $\spec B$.
It is an algebraic set.
For a Lipschitz singular simplex $\sigma:\Delta^n\to\spec B$ one introduces the dg-algebra morphism
$$\mu(\sigma):\Omega^\bullet_{B|\R}\to\Omega^\bullet_\flat (\Delta ^n),$$
see \cite[Section $6.1$]{baskov23}.

\begin{lemma}\label{lem:muiscompatible}
Take $h:[m]\to [n]$ a morphism in the category $\boldsymbol{\Delta}$.
The following diagram is commutative:
\begin{equation*}
\begin{tikzcd}
\Omega ^\bullet _{B|\R} \arrow[d,"\mu (\sigma)"']\arrow[rd,"\mu (h^*\sigma)"] &\\
\Omega ^\bullet _\flat (\Delta ^n)\arrow[r,"h^*"] & \Omega ^\bullet _\flat (\Delta ^{n-1})
.\end{tikzcd}
\end{equation*}
\end{lemma}

\begin{proof}
The proof is analogous to the proof of~\cite[Lemma~$21$]{baskov23}. 
\end{proof}

The lemma allows us to construct a dg-algebra morphism
$$\hat{\mu}:\Omega^\bullet _{B|\R}\to \Omega^\bullet _\flat(\Sing _\Lip (\spec B))$$
as 
\begin{equation}\label{form:mu}
\hat{\mu}(\omega)(\sigma)=\mu(\sigma)(\omega)
\end{equation}
for any $\omega\in\Omega^\bullet _{B|\R}$ and any $\sigma\in\Sing _\Lip (\spec B)_n$.

Any simplex $\beta\in\Delta[n]_m$ gives rise to a affine singular simplex $\Delta^m\to \Delta^n$.
This yields the restriction morphism of simplicial complexes $$\tau:\Omega^\bullet _\flat\to\CPL ^\bullet,$$
which is not multiplicative.

We have a map of complexes
$$\xi _B:\Omega^\bullet _{B|\R}\to \mathrm{S}^\bullet _\Lip(\spec B)$$
defined as
\begin{equation}\label{form:xi}
\xi_B(\omega)(\sigma) = \tau(\mu(\sigma)(\omega))(\Id_{[n]}),
\end{equation}
where $\Id_{[n]}\in\Delta[n]_n$, see \cite[Section $6.2$]{baskov23}.

The following diagram is commutative by Formula~\ref{form:mu} and Formula~\ref{form:xi}:
\begin{equation}\label{diag:taumuxiconnection}
\begin{tikzcd}[column sep = 8em, row sep = 3em]
\Omega^\bullet _\flat(\Sing_\Lip(\spec B))\ar[r,"\tau(\Sing_\Lip(\spec B))"] & \CPL ^\bullet(\Sing_\Lip(\spec B))\ar[d,Equal]\\
\Omega^\bullet _{B|\R}\ar[u,"\hat{\mu}"]\ar[r,"\xi_B",swap]&\CS^\bullet _\Lip(\spec B).
\end{tikzcd}
\end{equation}

\begin{lemma}\label{lem:multiplicativefingen}
The map
$$\HS^*(\xi _B):\HS^*(\Omega^\bullet _{B|\R})\to \HS^*(\mathrm{S}^\bullet _\Lip(\spec B))$$
is multiplicative.
\end{lemma}
The proof follows \cite[Theorem~$10.9$]{felix01}.
\begin{proof}
Consider the following commutative diagram of simplicial dg-algebras and simplicial cochains maps:
\begin{equation*}
\begin{tikzcd}[column sep=6em,row sep=4em]
\Omega^\bullet _\flat \ar[r,"e_1:\omega\mapsto\omega\otimes 1"]\ar[ddr,"\tau",swap] & \Omega^\bullet _\flat \otimes \CPL ^\bullet \ar[d,"\tau\otimes\Id"] & \CPL ^\bullet \ar[l,"1\otimes f\mapsfrom f:e_2",swap]\ar[ddl,"\Id",bend left]\\
&\CPL ^\bullet  \otimes \CPL ^\bullet \ar[d,"\mult"]&\\
&\CPL ^\bullet,&
\end{tikzcd}
\end{equation*}
where $\mult$ is the multiplication of $\CPL ^\bullet$.
Taking tensor product of coaugmentations on $\Omega^\bullet _\flat$ and $\CPL ^\bullet$ we obtain coaugmentations on $\Omega^\bullet _\flat \otimes \CPL ^\bullet$ and $\CPL ^\bullet \otimes \CPL ^\bullet$.
By Lemma~\ref{lem:lipschitzcochainsaredwcontractible}, Lemma~\ref{lem:omegaisdegreewisecontractible} and Corollary~\ref{cor:productisdegreewisecontractible} all the simplicial dg-algebras in this diagram are degree-wise contractible.
By Lemma~\ref{lem:lipschitzcochainsaredwcontractible}, Lemma~\ref{lem:omegaflatisacyclic} and the fact that all the cochain maps in the diagram preserve coaugmentation, all cochain maps are dimension-wise quasi-isomorphism.
Notice, that the maps $e_1$ and $e_2$ are in fact morphisms of simplicial dg-algebras.
Let $\eta:\Omega^\bullet _\flat \otimes \CPL ^\bullet\to\CPL ^\bullet$ be the composition of the vertical arrows.

Substituting a simplicial set $K$ gives rise to the commutative diagram of dg-algebras and cochains maps
\begin{equation*}
\begin{tikzcd}[column sep=6em,row sep=4em]
\Omega^\bullet _\flat(K) \ar[r,"e_1(K):\omega\mapsto\omega\otimes 1"]\ar[dr,"\tau(K)",swap] & (\Omega^\bullet _\flat \otimes \CPL ^\bullet)(K)\ar[d,"\eta(K)"] & \CPL ^\bullet(K)\ar[l,"1\otimes f\mapsfrom f:e_2(K)",swap]\ar[dl,"\Id",bend left]\\
&\CPL ^\bullet(K),&
\end{tikzcd}
\end{equation*}
where all of the arrows are quasi-isomorphisms by Lemma~\ref{lem:degreewisecontrinduces}.
The maps $e_1(K)$ and $e_2(K)$ are morphisms of dg-algebras.

Take the cohomology.
The map $\HS^*(e_2(K))$ is an isomorphism.
Since $\HS^*(\eta(K))\circ\HS^*(e_2(K))$ is the identity, we have $\HS^*(\eta(K))=(\HS^*(e_2(K)))^{-1}$.
We have
$$\HS^*(\tau (K))=\HS^*(\eta (K))\circ \HS^*(e_1(K))=(\HS^*(e_2(K)))^{-1}\circ \HS^*(e_1(K)),$$
hence the map $\HS^*(\tau (K))$ is multiplicative. 

Substitute $K=\Sing_\Lip(\spec B)$ and together with Diagram~\ref{diag:taumuxiconnection} one concludes that the map
$$\HS^*(\xi _B):\HS ^*(\Omega^\bullet _{B|\R})\to\HS^*(\CS ^\bullet _{\Lip}(\spec B))$$
is multiplicative.
\end{proof}
We denote by ${}^+\CS^\bullet _{\Lip,\spec B}$ the complex of sheaves of Lipschitz singular cochains on $\spec B$.
It can be equipped with the Alexander--Whitney cup product, which is non-commutative.

\subsection{The algebra of continuous functions.}
\begin{theorem}\label{thm:psiismultiplicative}
Let $X$ be a topological space, $A\subset C(X)$ be a subalgebra of the algebra of real-valued continuous functions on $X$.
Then the map
$$\Psi _A:\HS^*(\Omega^\bullet_{A|\R})\to \HH^*(X,\underline{\R}_X[0])$$
is multiplicative.
\end{theorem}
\begin{proof}

Take a finitely generated subalgebra $B\subset A$.
One has a continuous function $\Gamma _B:X\to\spec B$.
There is a morphism of dg-algebras
$$\sh:\CS^\bullet _{\Lip}(\spec B) \to {}^+\CS^\bullet _{\Lip,\spec B}(\spec B)$$
natural in subalgebra $B$.

There is also a morphism of sheaf complexes with multiplication
$$\epsilon:\RU _{\spec B}[0]\to {}^+\CS^\bullet_{\Lip,\spec B}$$
which is a quasi-isomorphism, see \cite[Lemma $3$]{baskov23}.

Consider the following maps.
\begin{enumerate}
\item
The map
$$\HS^*(\Omega^\bullet _{B|\R})\xrightarrow{\HS^*(\xi _B)}\HS^*(\mathrm{S}^\bullet _\Lip(\spec B))$$
is multiplicative by Lemma~\ref{lem:multiplicativefingen}.
\item
The map
$$\HS^*(\mathrm{S}^\bullet _\Lip(\spec B))\xrightarrow{\HS^*(\sh)}\HS^*({}^+\CS^\bullet _\Lip(\spec B))$$
is a clearly multiplicative.
\item
The map
$$\HS^*({}^+\CS^\bullet _\Lip(\spec B))\xrightarrow{\Upsilon}\HH^*(\spec B,{}^+\CS^\bullet_{\Lip,\spec B})$$
is multiplicative by the property~\ref{prop:compatiblewithupsilon} of the internal cup product.
\item
The map
$$\HH^*(\spec B,{}^+\CS^\bullet_{\Lip,\spec B})\xrightarrow{(\HH^*(\epsilon))^{-1}}\HH^*(\spec B,\RU _{\spec B}[0])$$
is multiplicative by the property~\ref{prop:naturalityofproduct} of the internal cup product.
\item
The map
$$\HH^*(\spec B,\RU _{\spec B}[0])\xrightarrow{\Gamma^*}\HH ^* (X,\RU _{X} [0])$$
is multiplicative by the property~\ref{prop:functorialitydifferentspaces} of the internal cup product.
\end{enumerate}
Hence the composition of the maps $(1)$--$(5)$
$$\HS^*(\Omega^\bullet _{B|\R})\to \HH ^* (X,\RU _{X} [0])$$
is multiplicative.

The map $\Psi _A$ is constructed by passing to the colimit of the above map over all finitely generated subalgebras $B\subset A$, hence 
is multiplicative.
\end{proof}

\Addresses

\end{document}